\documentclass[12pt]{amsart}

\usepackage[T1]{fontenc}
\usepackage[utf8]{inputenc}
\usepackage{amsmath,amssymb,amsthm,mathrsfs}
\usepackage{enumitem}
\usepackage[hidelinks]{hyperref}
\hypersetup{
  pdftitle={Combinatorial Proofs for Two-Color Partition Identities Involving Overpartitions and Minimal Excludants},
  pdfauthor={Dandan Chen, Rong Chen, and Mengjie Zhao},
  pdfkeywords={two-color partitions, overpartitions, minimal excludants, parity-reversing involutions, bijective proofs}
}

\newtheorem{theorem}{Theorem}[section]
\newtheorem{lemma}[theorem]{Lemma}

\theoremstyle{definition}
\newtheorem{definition}[theorem]{Definition}

\theoremstyle{remark}
\newtheorem{remark}[theorem]{Remark}
\newtheorem*{remark*}{Remark}
\usepackage{xcolor}
\numberwithin{equation}{section}
\allowdisplaybreaks

\title[Two-color partitions and overpartitions]{Combinatorial Proofs for Two-Color Partition Identities Involving Overpartitions and Minimal Excludants}
\author{Dandan Chen}
\address{Department of Mathematics, Shanghai University, People's Republic of China}
\address{Newtouch Center for Mathematics of Shanghai University, Shanghai, People's Republic of China}
\email{mathcdd@shu.edu.cn}

\author{Mengjie Zhao}
\address{Department of Mathematics, Shanghai University, People's Republic of China}
\email{zmjj@shu.edu.cn}
\author{Ziyin Zou}
\address{Department of Mathematics, Shanghai University, People's Republic of China}
\email{ziyinzou@126.com}


\begin{document}

\begin{abstract}
Andrews and El Bachraoui investigated two-color partitions in which even parts may occur only in blue and obtained several identities relating parity refinements of these partitions to overpartitions and minimal excludants. Some of their proofs are analytic and rely on $q$-series manipulations. In this paper, we construct explicit weight-preserving bijections and parity-reversing involutions that provide combinatorial proofs of these identities. We also give a new combinatorial proof of related overpartition refinements involving $\overline{p}(n)$ and $\overline{p}_o(n)$.
\end{abstract}

\maketitle

\section{Introduction}

In 2004, Corteel and Lovejoy \cite{CL_04} introduced overpartitions of a positive integer $n$. An overpartition is a partition in which the first occurrence of each distinct part may be overlined. Let $\overline{p}(n)$ denote the number of overpartitions of $n$. Hirschhorn and Sellers \cite{HS_06} studied overpartitions into odd parts; denoting their number by $\overline{p}_o(n)$, one has
\begin{equation*}
\sum_{n\ge0}\overline{p}(n)q^n=\frac{(-q;q)_\infty}{(q;q)_\infty},
\qquad
\sum_{n\ge0}\overline{p}_o(n)q^n=\frac{(-q;q^2)_\infty}{(q;q^2)_\infty}.
\end{equation*}
Here the $q$-shifted factorial \cite{A-1998} is defined by
\begin{equation*}
(a;q)_0=1,
\qquad
(a;q)_n=\prod_{j=0}^{n-1}(1-aq^j),
\qquad
(a;q)_\infty=\prod_{j=0}^{\infty}(1-aq^j),
\end{equation*}
where $|q|<1$ in the infinite product.

Throughout this paper, calligraphic letters denote sets of partitions, and the corresponding non-calligraphic letters denote their cardinalities. We write $|\pi|$ for the weight of a partition $\pi$ and $\ell(\pi)$ for its number of parts. Let $\overline{\mathcal P}_o(n)$ be the set of overpartitions of $n$ in which the overlined parts are distinct odd integers and the non-overlined parts are unrestricted odd integers. Thus $\overline{p}_o(n)=|\overline{\mathcal P}_o(n)|$.

Following Andrews and El Bachraoui \cite{AM-25}, we consider sequences of integer partitions in two colors (red and blue) with conditions on the parity and color of their summands. We write $\lambda_{\mathrm b}$ and $\lambda_{\mathrm r}$, respectively, for a part $\lambda$ occurring in blue and red, respectively, and use the order convention $\lambda_{\mathrm b}>\lambda_{\mathrm r}$ when equal numerical parts of different colors are displayed. The next two definitions are quoted in the form of \cite[Definitions 1 and 2]{AM-25}, with calligraphic notation used only to distinguish sets from their cardinalities.

\begin{definition}[{\cite[Definition 1]{AM-25}}]\label{def:F}
For any nonnegative integer $n$, let $\mathcal F(n)$ be the set of two-color integer partitions of $n$ such that the even parts may occur only in blue and let $F(n)=|\mathcal F(n)|$. Then we clearly have
\begin{equation*}
\sum_{n\ge0}F(n)q^n=\frac{1}{(q;q^2)_\infty^2(q^2;q^2)_\infty}
=1+2q+4q^2+8q^3+14q^4+24q^5+\cdots.
\end{equation*}
\end{definition}

\begin{definition}[{\cite[Definition 2]{AM-25}}]\label{def:F-refinements}
For a nonnegative integer $n$, let $F_0(n)$ (resp.\ $F_1(n)$) be the number of partitions in $\mathcal F(n)$ in which the number of odd parts in red is even (resp.\ odd). Furthermore, let $F_2(n)$ (resp.\ $F_3(n)$) be the number of partitions in $\mathcal F(n)$ in which the number of even parts is even (resp.\ odd).
\end{definition}

For later use, we write $r(\pi)$ for the number of odd parts of $\pi$ in red, counted with multiplicity, and $e(\pi)$ for the number of even parts of $\pi$, counted with multiplicity. Thus the corresponding subsets of $\mathcal F(n)$ are
\begin{align*}
\mathcal F_0(n)&=\{\pi\in\mathcal F(n): r(\pi)\equiv0\pmod2\},\\
\mathcal F_1(n)&=\{\pi\in\mathcal F(n): r(\pi)\equiv1\pmod2\},\\
\mathcal F_2(n)&=\{\pi\in\mathcal F(n): e(\pi)\equiv0\pmod2\},\\
\mathcal F_3(n)&=\{\pi\in\mathcal F(n): e(\pi)\equiv1\pmod2\},
\end{align*}
and $F_i(n)=|\mathcal F_i(n)|$ for $0\le i\le3$.

Andrews and El Bachraoui \cite{AM-25} obtained analytic proofs of several identities for these parity refinements. The first pair concerns the parity of the number of odd parts in red.

\begin{theorem}[{\cite[Theorem 5]{AM-25}}]\label{thm:F01}
For any nonnegative integer $n$, we have
\begin{align*}
F_0(n)&=\frac{\overline{p}(n)+\overline{p}(n/2)}{2},\\
F_1(n)&=\frac{\overline{p}(n)-\overline{p}(n/2)}{2}.
\end{align*}
\end{theorem}
Here and throughout, $\overline{p}(n/2)$ is understood to be $0$ when $n$ is odd.

To state the identities involving the parity of the number of even parts, we recall a minimal-excludant refinement. Andrews and Newman \cite{AN_20} defined the minimal excludant function $\operatorname{mex}_{A,a}(\pi)$. Andrews and El Bachraoui \cite{AM-25} extended this function to the present two-color setting as follows. Suppose that $A$ is even and that parts congruent to $0\pmod{A/2}$ may occur only in blue. For $\pi\in\mathcal F(n)$, let $\operatorname{mex}_{A,a}(\pi,\mathrm{blue})$ be the smallest positive integer congruent to $a\pmod A$ that does not occur in $\pi$. Let $p_{A,a}(n,\mathrm{blue})$ count the partitions $\pi\in\mathcal F(n)$ for which
\begin{equation*}
\operatorname{mex}_{A,a}(\pi,\mathrm{blue})\equiv a\pmod{2A},
\end{equation*}
and let $\overline{p}_{A,a}(n,\mathrm{blue})$ count the partitions $\pi\in\mathcal F(n)$ for which
\begin{equation*}
\operatorname{mex}_{A,a}(\pi,\mathrm{blue})\equiv A+a\pmod{2A}.
\end{equation*}
With this notation, the identities for $F_2(n)$ and $F_3(n)$ are as follows.

\begin{theorem}[{\cite[Theorems 3 and 4]{AM-25}}]\label{thm:F23}
For any nonnegative integer $n$, we have
\begin{align}
F_2(n)&=p_{4,2}(n,\mathrm{blue}),\label{eq:F2-mex}\\
F_3(n)&=\overline{p}_{4,2}(n,\mathrm{blue}).\label{eq:F3-mex}
\end{align}
\end{theorem}

To give a combinatorial proof of Theorem \ref{thm:F23}, we provide an explicit bijective proof of the following refinement of Euler's partition theorem due to Kang, Li, and Wang. Recently, Wang and  Zhang \cite{Wang} also gave an elegant combinatorial proof of this result.
\begin{lemma}[{\cite[Theorem 1.10]{KA_23}}]\label{lem:KLW 1}
The partitions of \(n\) into \(m\) odd parts whose minimal odd excludant is congruent to \(1\) \textnormal{(}respectively, \(3\)\textnormal{)} modulo \(4\) are equinumerous with the distinct partitions of \(n\) into an even \textnormal{(}respectively, odd\textnormal{)} number of parts with alternating sum \(m\).
\end{lemma}

In \cite{CZ_26}, Chen and Zou studied $F(n)$ and established identities expressing $F_2(n)$ and $F_3(n)$ ($F_0(n)$ and $F_1(n)$ in their notation), respectively, in terms of overpartitions. Written in the notation of Definition~\ref{def:F-refinements}, the relevant identities are as follows.

\begin{theorem}[{\cite[Theorem 1.3(a), (b), and (c)]{CZ_26}}]\label{thm:PO}
For every positive integer $n$,
\begin{enumerate}[label=\textnormal{(\alph*)},leftmargin=2.2em]
\item $F(n)=\overline{p}(n)$;
\item $F_2(n)=\frac{1}{2}\bigl(\overline{p}(n)+\overline{p}_o(n)\bigr)$;
\item $F_3(n)=\frac{1}{2}\bigl(\overline{p}(n)-\overline{p}_o(n)\bigr)$.
\end{enumerate}
\end{theorem}
Chen and Zou \cite{CZ_26} also gave a combinatorial proof of Theorem \ref{thm:PO}\textnormal{(a)} by constructing an explicit bijection. A combinatorial proof of Theorem \ref{thm:PO}\textnormal{(b), (c)} was recently obtained by Shen \cite{Shen}. In the present paper, we give combinatorial proofs of Theorems \ref{thm:F01} and \ref{thm:F23}, and we present another combinatorial proof of Theorem \ref{thm:PO}\textnormal{(b), (c)}. We do not repeat the known combinatorial proof of Theorem \ref{thm:PO}\textnormal{(a)}.

We next turn to the distinct-part analogue. The following definition records the underlying class from \cite[Definition 1]{AM_arxix}.

\begin{definition}[{\cite[Definition 1]{AM_arxix}}]\label{def:E}
Let $n$ be a nonnegative integer and let $E(n)$ be the number of two-color partitions of $n$ in which parts are distinct and the even parts may occur only in blue.
\end{definition}

In this paper, $\mathcal E(n)$ denotes the set counted by $E(n)$; thus $E(n)=|\mathcal E(n)|$. Equivalently,
\begin{equation*}
\sum_{n\ge0}E(n)q^n=(-q;q^2)_\infty^2(-q^2;q^2)_\infty.
\end{equation*}
The identity relating this class to overpartitions into odd parts is the following; a combinatorial proof was given by Chen and Zou \cite{CZ_26}. We shall use this identity as a known result and will not reprove it here.

\begin{theorem}[{\cite[Theorem 1.1(a)]{CZ_26}}]\label{thm:E-total}
For every nonnegative integer $n$, one has $E(n)=\overline{p}_o(n)$.
\end{theorem}

\begin{remark}\label{rem:gf-identities}
The identities $F(n)=\overline{p}(n)$ and $E(n)=\overline{p}_o(n)$ are also visible at the level of generating functions. By Euler's formula \cite{A-1998},
\begin{equation*}
(-q;q)_\infty=\frac{1}{(q;q^2)_\infty},
\end{equation*}
and hence
\begin{align*}
\sum_{n\ge0}F(n)q^n
&=\frac{1}{(q;q^2)_\infty^2(q^2;q^2)_\infty}
 =\frac{(-q;q)_\infty}{(q;q)_\infty}
 =\sum_{n\ge0}\overline{p}(n)q^n,\\
\sum_{n\ge0}E(n)q^n
&=(-q;q^2)_\infty^2(-q^2;q^2)_\infty
 =\frac{(-q;q^2)_\infty}{(q;q^2)_\infty}
 =\sum_{n\ge0}\overline{p}_o(n)q^n.
\end{align*}
\end{remark}

We now introduce the refinement of $E(n)$ used in this paper. It records the parity of the number of red odd parts. Throughout the rest of this paper, the symbols $E_0(n)$ and $E_1(n)$ refer to the following definition.

\begin{definition}\label{def:E-refinements}
For $\lambda\in\mathcal E(n)$, let $r(\lambda)$ be the number of red odd parts of $\lambda$. Define
\begin{align*}
\mathcal E_0(n)&=\{\lambda\in\mathcal E(n): r(\lambda)\equiv0\pmod2\},\\
\mathcal E_1(n)&=\{\lambda\in\mathcal E(n): r(\lambda)\equiv1\pmod2\},
\end{align*}
and let $E_i(n)=|\mathcal E_i(n)|$ for $i=0,1$.
\end{definition}

Our final companion result is the following red-odd-parity refinement of Theorem \ref{thm:E-total}.

\begin{theorem}\label{thm:E01}
For every positive integer $n$,
\begin{equation*}
E_0(n)=E_1(n)=\frac{\overline{p}_o(n)}{2}.
\end{equation*}
\end{theorem}
We prove this refinement by constructing a direct parity-reversing involution on $\mathcal E(n)$. Since
\[
\mathcal E(n)=\mathcal E_0(n)\,\dot\cup\,\mathcal E_1(n),
\]
and Theorem~\ref{thm:E-total} gives $E(n)=\overline p_o(n)$, this involution implies the stated formula for $E_0(n)$ and $E_1(n)$.

The paper is organized as follows. In Section \ref{sec:F01}, we give a combinatorial proof of Theorem \ref{thm:F01}. In Section \ref{sec:F23}, we prove Lemma \ref{lem:KLW 1}, Theorem \ref{thm:F23}, and then give a combinatorial proof of Theorem \ref{thm:PO}\textnormal{(b), (c)} using the sign-reversing involution from the first version of our argument. In Section \ref{sec:E01}, we prove Theorem \ref{thm:E01} by a direct parity-reversing involution on $\mathcal E(n)$.

\section{Proof of Theorem \ref{thm:F01}}\label{sec:F01}

We first recall the auxiliary class that appears naturally in the parity-reversing involution.  For a colored partition \(\lambda\), write \(\operatorname{mult}_{u_{\mathrm r}}(\lambda)\) and \(\operatorname{mult}_{u_{\mathrm b}}(\lambda)\) for the multiplicities of the red and blue copies of \(u\), respectively.

\begin{definition}\label{def:Q}
Let \(\mathcal Q(n)\) be the set of partitions \(\lambda\in\mathcal F(n)\) such that every blue part is even and every red odd part occurs with even multiplicity.  Let \(Q(n)=|\mathcal Q(n)|\).
\end{definition}

Thus \(\mathcal Q(n)\subseteq\mathcal F_0(n)\).  We shall identify \(\mathcal Q(n)\) with a simple class of overpartitions into even parts.

\begin{definition}\label{def:Pe}
Let \(\overline{\mathcal P}_e(n)\) be the set of overpartitions of \(n\) in which the overlined parts are distinct even integers and the non-overlined parts are unrestricted even integers.  Let \(\overline p_e(n)=|\overline{\mathcal P}_e(n)|\).
\end{definition}

\begin{lemma}\label{lem:Q-Pe}
For every nonnegative integer \(n\),
\[
Q(n)=\overline p_e(n).
\]
\end{lemma}

\begin{proof}
We construct mutually inverse weight-preserving maps
\[
\phi:\overline{\mathcal P}_e(n)\longrightarrow\mathcal Q(n),
\qquad
\psi:\mathcal Q(n)\longrightarrow\overline{\mathcal P}_e(n).
\]
Let \(\omega\in\overline{\mathcal P}_e(n)\).  The image \(\phi(\omega)=\lambda\) is obtained as follows.

\begin{enumerate}[label=\textnormal{(\roman*)},leftmargin=2.5em]
\item Each non-overlined even part \(m\) of \(\omega\) is sent to one blue part \(m_{\mathrm b}\).
\item Each overlined even part \(\overline m\) is written uniquely as
\[
m=2^a u,
\qquad a\ge1,
\qquad u\ \text{odd}.
\]
Replace \(\overline m\) by \(2^a\) red copies of \(u\), namely by \(u_{\mathrm r}^{2^a}\).  Equivalently, this is obtained by repeatedly splitting \(\overline m\) into two equal halves until all resulting parts are odd, and then coloring the resulting odd parts red.
\end{enumerate}

All blue parts produced are even.  Moreover, each overlined even part contributes an even number of red copies of an odd integer, so every red odd part in \(\lambda\) has even multiplicity.  Hence \(\lambda\in\mathcal Q(n)\), and the total weight is unchanged.

Conversely, let \(\lambda\in\mathcal Q(n)\).  The image \(\psi(\lambda)=\omega\) is constructed as follows.

\begin{enumerate}[label=\textnormal{(\roman*)},leftmargin=2.5em]
\item Each blue even part \(m_{\mathrm b}\) of \(\lambda\) is sent to a non-overlined part \(m\).
\item For each odd integer \(u\), let \(M_u=\operatorname{mult}_{u_{\mathrm r}}(\lambda)\).  Since \(\lambda\in\mathcal Q(n)\), \(M_u\) is even.  Write its binary expansion in the form
\[
M_u=\sum_{a\ge1}\varepsilon_{u,a}2^a,
\qquad \varepsilon_{u,a}\in\{0,1\}.
\]
For every \(a\) with \(\varepsilon_{u,a}=1\), insert the overlined even part \(\overline{2^a u}\) into \(\omega\).  This is the same operation as repeatedly merging two equal red parts into one part of twice the size until no repeated overlined part remains.
\end{enumerate}

The overlined parts obtained in this way are distinct, by uniqueness of binary expansion, and all parts of \(\omega\) are even.  The two constructions are inverse to each other, because the decomposition \(m=2^a u\) with \(u\) odd and the binary expansion of each \(M_u\) are unique.  Therefore \(\phi\) is a bijection, and \(Q(n)=\overline p_e(n)\).
\end{proof}

\begin{lemma}\label{lem:F01-difference}
For every nonnegative integer \(n\),
\[
F_0(n)-\overline p_e(n)=F_1(n).
\]
\end{lemma}

\begin{proof}
By Lemma \ref{lem:Q-Pe}, it suffices to prove
\[
F_0(n)-Q(n)=F_1(n).
\]
We construct an involution on
\(\mathcal F(n)\setminus\mathcal Q(n)\) which changes the parity of the number of red odd parts.

Let \(\pi_1=\emptyset\) and \(\pi_2=\emptyset\).  For
\[
\lambda=(\lambda_1^{c_1,i_1},\lambda_2^{c_2,i_2},\ldots,\lambda_t^{c_t,i_t})
\in \mathcal F(n)\setminus\mathcal Q(n),
\]
where \(c_j\in\{\mathrm r,\mathrm b\}\) and \(i_j\) is the multiplicity of the part \(\lambda_j\) in color \(c_j\), define \(\Phi(\lambda)=\lambda'\) by the following four steps.

\medskip
\noindent\textnormal{(S1)} For each \(j\in\{1,2,\ldots,t\}\), distribute the parts of \(\lambda\) into \(\pi_1\) and \(\pi_2\) as follows:
\begin{enumerate}[label=\textnormal{(\roman*)},leftmargin=2.5em]
\item if \(c_j=\mathrm b\) and \(\lambda_j\) is odd, put all \(i_j\) copies of \((\lambda_j)_{\mathrm b}\) into \(\pi_1\);
\item if \(c_j=\mathrm b\) and \(\lambda_j\) is even, put all \(i_j\) copies of \((\lambda_j)_{\mathrm b}\) into \(\pi_2\);
\item if \(c_j=\mathrm r\) and \(i_j\) is odd, put one copy of \((\lambda_j)_{\mathrm r}\) into \(\pi_1\) and put the remaining \(i_j-1\) copies into \(\pi_2\);
\item if \(c_j=\mathrm r\) and \(i_j\) is even, put all \(i_j\) copies of \((\lambda_j)_{\mathrm r}\) into \(\pi_2\).
\end{enumerate}
Here the red parts are necessarily odd, since even parts occur only in blue.  Since \(\lambda\notin\mathcal Q(n)\), either a blue odd part occurs in \(\lambda\), or a red odd part occurs with odd multiplicity.  Hence \(\pi_1\neq\emptyset\).

\medskip
\noindent\textnormal{(S2)} Rearrange \(\pi_1\) in the ordered form
\[
\pi_1=(\mu_1^{d_1,k_1},\mu_2^{d_2,k_2},\ldots,\mu_s^{d_s,k_s}),
\]
where the parts are ordered increasingly by size, and red parts precede blue parts when their sizes are equal.

\medskip
\noindent\textnormal{(S3)} Change the color of exactly one copy of the first part \(\mu_1^{d_1}\).  Denote the resulting colored partition by \(\pi_1'\).

\medskip
\noindent\textnormal{(S4)} Combine \(\pi_1'\) with \(\pi_2\), and then rearrange the result into ordered form.  The resulting colored partition is denoted by \(\lambda'\).  We set
\[
\Phi(\lambda)=\lambda'.
\]

The construction preserves weight and maps \(\mathcal F(n)\setminus\mathcal Q(n)\) to itself: a selected red part becomes a blue odd part, or a selected blue part becomes a red part of odd multiplicity.  If \(m\) is the value selected in \textnormal{(S3)}, then all smaller candidates remain unchanged; hence the next application selects the same value with the opposite color and changes it back. Thus \(\Phi^2=\mathrm{id}\). The number of red odd parts changes by one, so \(\Phi\) pairs \(\mathcal F_0(n)\setminus\mathcal Q(n)\) with \(\mathcal F_1(n)\). Hence
\[
F_0(n)-Q(n)=F_1(n),
\]
and the lemma follows.
\end{proof}

\begin{proof}[Proof of Theorem \ref{thm:F01}]
Every partition in \(\mathcal F(n)\) is counted by exactly one of \(F_0(n)\) and \(F_1(n)\), so
\[
F_0(n)+F_1(n)=F(n).
\]
By Lemma \ref{lem:F01-difference},
\[
F_0(n)-F_1(n)=\overline p_e(n).
\]
Solving these two equations gives
\[
F_0(n)=\frac{F(n)+\overline p_e(n)}2,
\qquad
F_1(n)=\frac{F(n)-\overline p_e(n)}2.
\]
The identity \(F(n)=\overline p(n)\) is Theorem \ref{thm:PO}\textnormal{(a)} for \(n>0\), and is immediate for \(n=0\).  Moreover, dividing every part of an overpartition in \(\overline{\mathcal P}_e(n)\) by \(2\) gives an overpartition of \(n/2\), and this operation is reversible.  Hence
\[
\overline p_e(n)=\overline p(n/2),
\]
with the convention that this number is \(0\) when \(n\) is odd.  Substitution yields
\[
F_0(n)=\frac{\overline p(n)+\overline p(n/2)}2,
\qquad
F_1(n)=\frac{\overline p(n)-\overline p(n/2)}2.
\]
This proves Theorem \ref{thm:F01}.
\end{proof}

\section{Proofs of Theorems \ref{thm:F23} and \ref{thm:PO}\textnormal{(b), (c)}}\label{sec:F23}

Let \(\mathcal P(n)\) be the set of ordinary partitions of \(n\). For \(\alpha\in\mathcal P(n)\), define the minimal odd excludant \(\operatorname{omex}(\alpha)\) to be the least positive odd integer that does not occur as a part of \(\alpha\).

\subsection{A bijective proof of Lemma~\ref{lem:KLW 1}}
\begin{definition}\cite[p.2]{konan}
  Let $j\in\mathbb{Z}_{\geq0}$. For $\lambda\in\mathcal{P}$, the $j$-mex of $\lambda$, denoted $\operatorname{mex}_j(\lambda)$, is the smallest integer greater than $j$ that is not a part of $\lambda$. Denote by $\mathcal{M}_j$ the set of partitions whose $j$-mex has parity different from that of $j$.
  Denote by $\mathcal{F}_j$ the set of partitions $\lambda$ such that $j\notin\{\lambda_i-i:i\in\{1,\cdots,\ell(\lambda)\}\}$.
\end{definition}

\begin{lemma}\cite[Theorem 1.10]{konan}\label{lem:K_0}
  For $j\in\mathbb{Z}_{\geq0}$, at a fixed weight, the number of partitions in $\mathcal{M}_j$ is equal to the number of partitions in $\mathcal{F}_j$.
\end{lemma}
\begin{lemma}\cite[Corollary 1.11]{konan}\label{lem:K_1}
  For $j\in\mathbb{Z}_{\geq0}$, at a fixed weight,
  the number of partitions
  in $\overline{\mathcal{M}}_j$
  is equal to the number of partitions in $\overline{\mathcal{F}}_j$.
\end{lemma}
For a partition
$$\beta=(\beta_1,\beta_2,\ldots,\beta_{\ell(\beta)})$$
written in decreasing order, define its alternating sum by
$$\operatorname{alt}(\beta)=\beta_1-\beta_2+\beta_3-\beta_4+\cdots
+(-1)^{\ell(\beta)-1}\beta_{\ell(\beta)}.$$

Let $\mathcal O(n,m)$ denote the set of partitions of $n$ into
exactly $m$ odd parts. For $j\in\{0,1\}$, define
\[
\mathcal O_j(n,m)
=
\left\{
\eta\in\mathcal O(n,m):
\operatorname{omex}(\eta)\equiv 2j+1\pmod 4
\right\}.
\]

Let $\mathcal D_j(n,m)$ denote the set of partitions $\beta$ of $n$
into distinct parts such that
\[
\operatorname{alt}(\beta)=m
\qquad\text{and}\qquad
\ell(\beta)\equiv j\pmod 2.
\]

We now give a combinatorial proof of Lemma \ref{lem:KLW 1}.

\begin{proof}[Proof of Lemma \ref{lem:KLW 1}]
Fix $j\in\{0,1\}$. We construct mutually inverse weight-preserving maps
\[
\Phi_j:\mathcal O_j(n,m)\longrightarrow\mathcal D_j(n,m)
\qquad\text{and}\qquad
\Psi_j:\mathcal D_j(n,m)\longrightarrow\mathcal O_j(n,m).
\]
When $n=m=0$, both sets consist only of the empty partition, and we set
$\Phi_j(\emptyset)=\emptyset$. Henceforth, assume that $m>0$.

Let
\[
\eta=(\eta_1,\eta_2,\ldots,\eta_m)\in\mathcal O_j(n,m),
\]
where the parts are written in decreasing order. We construct
$\Phi_j(\eta)=\beta$ in four steps.

\medskip
\noindent\textnormal{(S1)} Define
\[
\sigma=H(\eta)=(\sigma_1,\sigma_2,\ldots,\sigma_m),
\qquad
\sigma_i=\frac{\eta_i+1}{2}\quad(1\le i\le m).
\]
Since every part of $\eta$ is odd, $\sigma$ is an ordinary partition with
\[
\ell(\sigma)=m,
\qquad
|\sigma|=\frac{n+m}{2}.
\]
Moreover,
\[
\operatorname{mex}(\sigma)
=\frac{\operatorname{omex}(\eta)+1}{2}.
\]
Thus $\operatorname{mex}(\sigma)$ is odd when $j=0$ and even when $j=1$.

\medskip
\noindent\textnormal{(S2)} We next apply the length- and weight-preserving bijections underlying Lemmas~\ref{lem:K_0} and~\ref{lem:K_1}.

If $j=0$, then $\sigma\in\mathcal M_0$. Applying the bijection underlying Lemma~\ref{lem:K_0}, we obtain a partition
\[
\rho=(\rho_1,\rho_2,\ldots,\rho_m)\in\mathcal F_0
\]
such that $|\rho|=|\sigma|$. In particular,
\[
\rho_i\ne i\qquad(1\le i\le m).
\]

If $j=1$, then $\sigma\in\overline{\mathcal M}_0$. The complementary bijection underlying Lemma~\ref{lem:K_1} first gives a pair $((1),\mu)$, where
\[
\mu=(\mu_1,\mu_2,\ldots,\mu_{m-1})\in\mathcal F_1,
\qquad
|\mu|=|\sigma|-1.
\]
Since $\mu_i-i\ne1$ for every $i$, let
\[
d=\max\{i:\mu_i>i+1\},
\]
with the convention that $d=0$ if this set is empty, and define
\[
\rho=(\mu_1-1,\mu_2-1,\ldots,\mu_d-1,
       d+1,\mu_{d+1},\mu_{d+2},\ldots,\mu_{m-1}).
\]
The maximality of $d$ and the condition $\mu\in\mathcal F_1$ imply that
\[
\mu_i>i+1\quad(1\le i\le d),
\qquad
\mu_i<i+1\quad(i>d).
\]
Consequently, $\rho$ is a partition of length $m$ satisfying
\[
|\rho|=|\mu|+1=|\sigma|,
\qquad
\rho_{d+1}=d+1,
\]
and $\rho_i\ne i$ for $i\ne d+1$. Thus $\rho\in\overline{\mathcal F}_0$ and has a unique fixed point.

\medskip
\noindent\textnormal{(S3)} Define
\[
\lambda=(\lambda_1,\lambda_2,\ldots,\lambda_m),
\qquad
\lambda_i=2\rho_i-1\quad(1\le i\le m).
\]
Then $\lambda$ is a partition into $m$ odd parts, and
\[
|\lambda|=2|\rho|-m=2|\sigma|-m=n.
\]
If $j=0$, then
\[
\lambda_i\ne2i-1\qquad(1\le i\le m),
\]
whereas if $j=1$, there is a unique index $i$ for which
$\lambda_i=2i-1$.

\medskip
\noindent\textnormal{(S4)} Let
\[
\lambda'=(\lambda_1',\lambda_2',\ldots,\lambda_r')
\]
be the conjugate partition of $\lambda$, and set
\[
I=\max\{i:\lambda_i\ge2i-1\}.
\]
Since the sequence $\lambda_i-(2i-1)$ is strictly decreasing, define
\[
\epsilon=
\begin{cases}
0,&\lambda_I>2I-1,\\
1,&\lambda_I=2I-1.
\end{cases}
\]
The conclusion of \textnormal{(S3)} shows that $\epsilon=0$ when $j=0$ and $\epsilon=1$ when $j=1$.

For $1\le i\le I$, set
\[
\beta_{2i-1}
=\frac{\lambda_i+1}{2}+\lambda'_{2i-1}-2i+1,
\]
and, for $1\le i\le I-\epsilon$, set
\[
\beta_{2i}
=\frac{\lambda_i-1}{2}+\lambda'_{2i+1}-2i+1.
\]
These formulas are the row-length form of the usual fish-hook construction. In particular, the Ferrers diagram of $\lambda$ is decomposed into successive hooks whose sizes are the parts of $\beta$; hence
\[
|\beta|=|\lambda|=n.
\]
For completeness, the strict decrease of the parts follows from
\begin{align*}
\beta_{2i-1}-\beta_{2i}
&=1+\lambda'_{2i-1}-\lambda'_{2i+1}>0
&& (1\le i\le I-\epsilon),\\
\beta_{2i}-\beta_{2i+1}
&=\frac{\lambda_i-\lambda_{i+1}+2}{2}>0
&& (1\le i\le I-1).
\end{align*}
The final part is positive: if $\epsilon=0$, then
$\lambda'_{2I+1}=I$; if $\epsilon=1$, then
$\lambda_I=2I-1$ and $\lambda'_{2I-1}\ge I$. Therefore, $\beta$ is a partition into distinct parts, and
\[
\ell(\beta)=2I-\epsilon\equiv j\pmod2.
\]

It remains to verify the alternating sum. If $\epsilon=0$, then
\begin{align*}
\operatorname{alt}(\beta)
&=\sum_{i=1}^{I}(\beta_{2i-1}-\beta_{2i})\\
&=I+\lambda'_1-\lambda'_{2I+1}
=m.
\end{align*}
If $\epsilon=1$, then
\[
\beta_{2I-1}
=\frac{\lambda_I+1}{2}+\lambda'_{2I-1}-2I+1
=\lambda'_{2I-1}-I+1,
\]
and hence
\begin{align*}
\operatorname{alt}(\beta)
&=\sum_{i=1}^{I-1}(\beta_{2i-1}-\beta_{2i})
  +\beta_{2I-1}\\
&=(I-1)+\lambda'_1-\lambda'_{2I-1}
  +\lambda'_{2I-1}-I+1\\
&=m.
\end{align*}
Thus $\beta\in\mathcal D_j(n,m)$, and $\Phi_j$ is well defined.

We now describe the inverse map. Starting with
$\beta\in\mathcal D_j(n,m)$, reverse the fish-hook construction in
\textnormal{(S4)} to recover $\lambda$. Next set
\[
\rho_i=\frac{\lambda_i+1}{2}
\qquad(1\le i\le m).
\]
When $j=0$, the partition $\rho$ lies in $\mathcal F_0$, and the inverse of the bijection in Lemma~\ref{lem:K_0} recovers $\sigma$. When $j=1$, the unique fixed point of $\rho$ is removed, and $1$ is added to every preceding part; this recovers the pair $((1),\mu)$, after which the inverse complementary bijection in Lemma~\ref{lem:K_1} recovers $\sigma$. Finally, set
\[
\eta_i=2\sigma_i-1
\qquad(1\le i\le m).
\]
This defines a weight-preserving map
\[
\Psi_j:\mathcal D_j(n,m)\longrightarrow\mathcal O_j(n,m).
\]
Each step is the inverse of the corresponding step in the construction of $\Phi_j$. Therefore,
\[
\Psi_j\circ\Phi_j=\operatorname{id}_{\mathcal O_j(n,m)}
\qquad\text{and}\qquad
\Phi_j\circ\Psi_j=\operatorname{id}_{\mathcal D_j(n,m)}.
\]
Hence $\Phi_j$ and $\Psi_j$ are mutually inverse weight-preserving bijections, which proves Lemma~\ref{lem:KLW 1}.
\end{proof}

\subsection{Bijective proofs of Theorems \ref{thm:F23} and \ref{thm:PO}\textnormal{(b), (c)}}

We also write \(\operatorname{mult}_u(\alpha)\) for the multiplicity of the part \(u\) in \(\alpha\).

For \(j\in\{0,1\}\), let
\begin{align*}
\mathcal A_j(n)&=\{\alpha\in\mathcal P(n):\operatorname{omex}(\alpha)\equiv 2j+1\pmod4\},\\
\mathcal B_j(n)&=\{\beta\in\mathcal P(n):\ell(\beta)\equiv j\pmod2\},
\end{align*}
and put \(A_j(n)=|\mathcal A_j(n)|\) and \(B_j(n)=|\mathcal B_j(n)|\). We shall also use the following set notation for the minimal-excludant refinements:
\begin{align*}
\mathcal P_{A,a}(n)&=\{\pi\in\mathcal F(n):\operatorname{mex}_{A,a}(\pi,\mathrm{blue})\equiv a\pmod{2A}\},\\
\overline{\mathcal P}_{A,a}(n)&=\{\pi\in\mathcal F(n):\operatorname{mex}_{A,a}(\pi,\mathrm{blue})\equiv A+a\pmod{2A}\}.
\end{align*}
Thus \(|\mathcal P_{A,a}(n)|=p_{A,a}(n,\mathrm{blue})\) and \(|\overline{\mathcal P}_{A,a}(n)|=\overline p_{A,a}(n,\mathrm{blue})\).

After summing over the statistic \(m\), Lemma \ref{lem:KLW 1} implies that, for each \(N\ge0\) and \(j\in\{0,1\}\), the odd partitions of \(N\) with minimal odd excludant congruent to \(2j+1\pmod4\) are equinumerous with the distinct partitions of \(N\) whose number of parts is congruent to \(j\) modulo \(2\). We fix, once and for all, a weight-preserving correspondence \(\rho_{j,N}\) between these two sets, and denote its inverse by \(\rho_{j,N}^{-1}\).

\begin{lemma}\label{lem:omex-parity}
For every nonnegative integer \(n\) and \(j\in\{0,1\}\),
\[
A_j(n)=B_j(n).
\]
\end{lemma}

\begin{proof}
For each fixed \(j\in\{0,1\}\), we define mutually inverse maps
\[
\phi_j:\mathcal A_j(n)\longrightarrow\mathcal B_j(n),
\qquad
\psi_j:\mathcal B_j(n)\longrightarrow\mathcal A_j(n).
\]

Let \(\alpha\in\mathcal A_j(n)\). We construct \(\phi_j(\alpha)=\beta\) as follows.

\smallskip
\noindent\textnormal{(S1)} Separate the even parts and the odd parts of \(\alpha\).

\smallskip
\noindent\textnormal{(S2)} Divide each even part by \(2\), and let \(E\) be the resulting ordinary partition. Let \(\eta\) be the partition formed by the odd parts of \(\alpha\). Then
\[
|\alpha|=2|E|+|\eta|,
\qquad
\operatorname{omex}(\eta)=\operatorname{omex}(\alpha)\equiv2j+1\pmod4.
\]

\smallskip
\noindent\textnormal{(S3)} Apply the fixed correspondence \(\rho_{j,|\eta|}\) to \(\eta\), and write
\[
\delta=\rho_{j,|\eta|}(\eta).
\]
Then \(\delta\) is a distinct partition of \(|\eta|\) and \(\ell(\delta)\equiv j\pmod2\).

\smallskip
\noindent\textnormal{(S4)} Define \(\beta\) by replacing each part of \(E\) by two copies of the same part and then adjoining the parts of \(\delta\). Equivalently, for each positive integer \(u\),
\[
\operatorname{mult}_u(\beta)=2\operatorname{mult}_u(E)+\mathbf 1_{\{u\in\delta\}}.
\]
Hence
\[
|\beta|=2|E|+|\delta|=2|E|+|\eta|=|\alpha|=n
\]
and
\[
\ell(\beta)=2\ell(E)+\ell(\delta)\equiv j\pmod2.
\]
Thus \(\beta\in\mathcal B_j(n)\).

Conversely, let \(\beta\in\mathcal B_j(n)\). We construct \(\psi_j(\beta)=\alpha\) as follows.

\smallskip
\noindent\textnormal{(S1)} For each positive integer \(u\), write
\[
\operatorname{mult}_u(\beta)=2e_u+\varepsilon_u,
\qquad e_u\ge0,
\qquad \varepsilon_u\in\{0,1\}.
\]
Separate the copies of each part \(u\) into \(e_u\) pairs and a possible remaining single copy.

\smallskip
\noindent\textnormal{(S2)} The paired copies form an ordinary partition \(E\), while the remaining single copies form a distinct partition \(\delta\). Thus
\[
|\beta|=2|E|+|\delta|,
\qquad
\ell(\delta)=\ell(\beta)-2\ell(E)\equiv j\pmod2.
\]

\smallskip
\noindent\textnormal{(S3)} Apply \(\rho_{j,|\delta|}^{-1}\) to \(\delta\), and write
\[
\eta=\rho_{j,|\delta|}^{-1}(\delta).
\]
Then \(\eta\) is an odd partition of \(|\delta|\) and
\[
\operatorname{omex}(\eta)\equiv2j+1\pmod4.
\]

\smallskip
\noindent\textnormal{(S4)} Define \(\alpha\) by doubling each part of \(E\) and adjoining the odd parts of \(\eta\). Then
\[
|\alpha|=2|E|+|\eta|=2|E|+|\delta|=|\beta|=n,
\]
and
\[
\operatorname{omex}(\alpha)=\operatorname{omex}(\eta)\equiv2j+1\pmod4.
\]
Thus \(\alpha\in\mathcal A_j(n)\).

The decomposition of \(\alpha\) into even and odd parts, the decomposition of \(\beta\) into paired copies and single copies, and the correspondence \(\rho_{j,N}\) are all reversible. Hence \(\psi_j\circ\phi_j\) and \(\phi_j\circ\psi_j\) are the identity maps on their respective domains. Therefore \(\phi_j\) is a bijection, and \(A_j(n)=B_j(n)\).
\end{proof}

\begin{proof}[Proof of Theorem \ref{thm:F23}]
For \(\pi\in\mathcal F(n)\), let \(e(\pi)\) be the number of even parts of \(\pi\). Since even parts may occur only in blue, this is the number of blue even parts of \(\pi\).

We first prove \eqref{eq:F2-mex}. It is enough to construct a bijection
\[
\Phi:\mathcal F_2(n)\longrightarrow\mathcal P_{4,2}(n).
\]
Let \(\pi\in\mathcal F_2(n)\). We construct \(\Phi(\pi)=\tau\) as follows.

\begin{enumerate}[leftmargin=2.4em]
\item If \(\pi\) has no even part, then \(2\) is the smallest positive integer congruent to \(2\pmod8\) which does not occur as a blue part. Hence \(\tau=\pi\) lies in \(\mathcal P_{4,2}(n)\).

\item Suppose that \(\pi\) has at least one even part. Separate the blue even parts from the colored odd parts of \(\pi\). Divide every blue even part by \(2\), and let \(\alpha\) be the resulting ordinary partition. Let \(\omega\) be the colored odd subpartition left unchanged. Then
\[
|\pi|=2|\alpha|+|\omega|,
\qquad
\ell(\alpha)=e(\pi)\equiv0\pmod2.
\]
Thus \(\alpha\in\mathcal B_0(|\alpha|)\). Applying the inverse map \(\psi_0\) from Lemma \ref{lem:omex-parity} to \(\alpha\), we obtain a partition \(\alpha'\in\mathcal A_0(|\alpha|)\); that is,
\[
|\alpha'|=|\alpha|,
\qquad
\operatorname{omex}(\alpha')\equiv1\pmod4.
\]
Now replace each part \(u\) of \(\alpha'\) by the blue part \((2u)_{\mathrm b}\), and then adjoin the colored odd subpartition \(\omega\). Denote the resulting two-color partition by \(\tau\). Then
\[
|\tau|=2|\alpha'|+|\omega|=2|\alpha|+|\omega|=|\pi|=n,
\]
and
\[
\operatorname{mex}_{4,2}(\tau,\mathrm{blue})
=2\operatorname{omex}(\alpha')
\equiv2\pmod8.
\]
Hence \(\tau\in\mathcal P_{4,2}(n)\).
\end{enumerate}

Conversely, define a map
\[
\Psi:\mathcal P_{4,2}(n)\longrightarrow\mathcal F_2(n)
\]
as follows. Let \(\tau\in\mathcal P_{4,2}(n)\).

\begin{enumerate}[leftmargin=2.4em]
\item If \(\tau\) has no even part, then \(\Psi(\tau)=\tau\). In this case \(e(\tau)=0\), so \(\tau\in\mathcal F_2(n)\).

\item Suppose that \(\tau\) has at least one even part. Separate the blue even parts from the colored odd parts of \(\tau\). Divide every blue even part by \(2\), and let \(\alpha'\) be the resulting ordinary partition. Let \(\omega\) be the colored odd subpartition left unchanged. Since
\[
\operatorname{mex}_{4,2}(\tau,\mathrm{blue})=2\operatorname{omex}(\alpha')
\]
and \(\tau\in\mathcal P_{4,2}(n)\), we have \(\operatorname{omex}(\alpha')\equiv1\pmod4\). Thus \(\alpha'\in\mathcal A_0(|\alpha'|)\). Applying the map \(\phi_0\) from Lemma \ref{lem:omex-parity} to \(\alpha'\), we obtain a partition \(\alpha\in\mathcal B_0(|\alpha'|)\); that is,
\[
|\alpha|=|\alpha'|,
\qquad
\ell(\alpha)\equiv0\pmod2.
\]
Replace each part \(u\) of \(\alpha\) by the blue part \((2u)_{\mathrm b}\), and then adjoin \(\omega\). Denote the resulting partition by \(\pi\). Then \(|\pi|=|\tau|=n\) and
\[
e(\pi)=\ell(\alpha)\equiv0\pmod2.
\]
Therefore \(\pi\in\mathcal F_2(n)\).
\end{enumerate}

The two constructions are inverse to each other, because the colored odd subpartition is kept fixed and the halved even subpartition is transformed by the mutually inverse maps \(\phi_0\) and \(\psi_0\) of Lemma \ref{lem:omex-parity}. Hence \(\Phi\) is a bijection and
\[
F_2(n)=p_{4,2}(n,\mathrm{blue}).
\]

The proof of \eqref{eq:F3-mex} is obtained by the same construction with \(j=1\). More precisely, the maps \(\phi_1\) and \(\psi_1\) in Lemma \ref{lem:omex-parity} give a bijection between \(\mathcal F_3(n)\) and \(\overline{\mathcal P}_{4,2}(n)\), since \(\operatorname{omex}(\alpha')\equiv3\pmod4\) is equivalent to
\[
\operatorname{mex}_{4,2}(\tau,\mathrm{blue})=2\operatorname{omex}(\alpha')\equiv6\pmod8.
\]
Thus
\[
F_3(n)=\overline p_{4,2}(n,\mathrm{blue}).
\]
\end{proof}

We next prove Theorem \ref{thm:PO}\textnormal{(b), (c)}. Let \(\mathcal G(n)\) be the set of partitions in \(\mathcal F(n)\) with no even parts, such that each odd red part occurs at most once, while each odd blue part may occur with arbitrary multiplicity. Let \(g(n)=|\mathcal G(n)|\).

\begin{lemma}\label{lem:G-po}
For every nonnegative integer \(n\),
\[
g(n)=\overline p_o(n).
\]
\end{lemma}

\begin{proof}
Define a map
\[
\phi:\overline{\mathcal P}_o(n)\longrightarrow\mathcal G(n)
\]
as follows. For \(\lambda\in\overline{\mathcal P}_o(n)\), let \(\phi(\lambda)=\mu\), where \(\mu\) is constructed by the following rules:
\begin{enumerate}[label=\textnormal{(\arabic*)},leftmargin=2.4em]
\item map each overlined odd part \(\overline u\) of \(\lambda\) to the red part \(u_{\mathrm r}\);
\item map each non-overlined odd part \(u\) of \(\lambda\) to the blue part \(u_{\mathrm b}\).
\end{enumerate}
The overlined parts in an overpartition are distinct, so the red odd parts in \(\mu\) are distinct; the non-overlined parts may occur with arbitrary multiplicity, so the blue odd parts in \(\mu\) may also occur with arbitrary multiplicity. Thus \(\mu\in\mathcal G(n)\), and the weight is preserved.

Conversely, define
\[
\psi:\mathcal G(n)\longrightarrow\overline{\mathcal P}_o(n)
\]
by sending each red odd part \(u_{\mathrm r}\) to the overlined part \(\overline u\), and each blue odd part \(u_{\mathrm b}\) to the non-overlined part \(u\). These two operations are inverse to each other. Hence \(\phi\) is a bijection, and \(g(n)=\overline p_o(n)\).
\end{proof}

\begin{proof}[Proof of Theorem \ref{thm:PO}\textnormal{(b), (c)}]
For \(\pi\in\mathcal F(n)\), let \(e(\pi)\) be the number of even parts of \(\pi\). Then
\begin{equation}\label{eq:signed-sum-F23}
F_2(n)-F_3(n)=\sum_{\pi\in\mathcal F(n)}(-1)^{e(\pi)}.
\end{equation}
We define a sign-reversing involution
\[
\Phi:\mathcal F(n)\setminus\mathcal G(n)\longrightarrow\mathcal F(n)\setminus\mathcal G(n).
\]
Let \(\pi\in\mathcal F(n)\setminus\mathcal G(n)\). For each positive odd integer \(k\), consider the chain
\[
k_{\mathrm r},\quad (2k)_{\mathrm b},\quad (4k)_{\mathrm b},\quad (8k)_{\mathrm b},\ldots .
\]
Here \(k_{\mathrm r}\) denotes a red part \(k\), and \((2^s k)_{\mathrm b}\) denotes a blue part \(2^s k\), for \(s\ge1\). The blue odd part \(k_{\mathrm b}\) is not included in this chain. Write
\[
m_0(k)=\operatorname{mult}_{k_{\mathrm r}}(\pi),
\qquad
m_s(k)=\operatorname{mult}_{(2^s k)_{\mathrm b}}(\pi)\quad (s\ge1).
\]
We call the \(k\)-chain good if
\[
m_0(k)\in\{0,1\},
\qquad
m_s(k)=0\quad(s\ge1).
\]
Since \(\pi\notin\mathcal G(n)\), at least one \(k\)-chain is not good. Let \(k\) be the least positive odd integer whose chain is not good. We construct \(\Phi(\pi)=\mu\) by changing only this \(k\)-chain.

\begin{enumerate}[leftmargin=2.4em]
\item Suppose that the \(k\)-chain contains at least one blue even part. Let
\[
\ell=\max\{s\ge1:m_s(k)>0\}.
\]
Thus \((2^\ell k)_{\mathrm b}\) is a highest blue even part in this chain.
\begin{enumerate}[label=\textnormal{(\alph*)},leftmargin=2.4em]
\item If \(m_\ell(k)=1\), replace one copy of \((2^\ell k)_{\mathrm b}\) by two copies of \(2^{\ell-1}k\). When \(\ell=1\), these two new parts are two red parts \(k_{\mathrm r}\); when \(\ell>1\), they are two blue even parts \((2^{\ell-1}k)_{\mathrm b}\).
\item If \(m_\ell(k)\ge2\), replace two copies of \((2^\ell k)_{\mathrm b}\) by one copy of \((2^{\ell+1}k)_{\mathrm b}\).
\end{enumerate}

\item Suppose that the \(k\)-chain contains no blue even part. Since the chain is not good, we must have \(m_0(k)\ge2\). In this case, replace two copies of \(k_{\mathrm r}\) by one copy of \((2k)_{\mathrm b}\).
\end{enumerate}

The resulting partition \(\mu\) has the same weight as \(\pi\) and remains in \(\mathcal F(n)\setminus\mathcal G(n)\).  All smaller chains are unchanged, so the next application selects the same \(k\)-chain and reverses the preceding split or merge. Hence \(\Phi^2=\mathrm{id}\). In every case the number of blue even parts changes by one modulo \(2\). Thus
\[
e(\Phi(\pi))\equiv e(\pi)+1\pmod2.
\]
Consequently, all elements of \(\mathcal F(n)\setminus\mathcal G(n)\) cancel in the signed sum \eqref{eq:signed-sum-F23}. The remaining elements are precisely the partitions in \(\mathcal G(n)\); they have no even parts and therefore contribute with positive sign. By Lemma \ref{lem:G-po},
\[
F_2(n)-F_3(n)=g(n)=\overline p_o(n).
\]

Since every element of \(\mathcal F(n)\) is counted by exactly one of \(F_2(n)\) and \(F_3(n)\), and since Theorem \ref{thm:PO}\textnormal{(a)} gives \(F(n)=\overline p(n)\), we also have
\[
F_2(n)+F_3(n)=F(n)=\overline p(n).
\]
Solving these two equations gives
\begin{align*}
F_2(n)&=\frac12\bigl(\overline p(n)+\overline p_o(n)\bigr),\\
F_3(n)&=\frac12\bigl(\overline p(n)-\overline p_o(n)\bigr).
\end{align*}
This proves Theorem \ref{thm:PO}\textnormal{(b), (c)}.
\end{proof}

\section{Proof of Theorem \ref{thm:E01}}\label{sec:E01}

In this section we use Theorem~\ref{thm:E-total} as a known result and prove only the refinement stated in Theorem~\ref{thm:E01}. Since
\[
\mathcal E(n)=\mathcal E_0(n)\,\dot\cup\,\mathcal E_1(n),
\]
it is enough to construct a parity-reversing involution on \(\mathcal E(n)\), where the parity is that of the number of red odd parts.

\begin{proof}[Proof of Theorem \ref{thm:E01}]
Let \(\lambda\in\mathcal E(n)\). For each positive odd integer \(v\), consider the chain
\[
v_{\mathrm r},\quad v_{\mathrm b},\quad (2v)_{\mathrm b},\quad (4v)_{\mathrm b},\quad (8v)_{\mathrm b},\ldots .
\]
Here \(v_{\mathrm r}\) and \(v_{\mathrm b}\) denote the red and blue copies of the odd part \(v\), respectively, while \((2^jv)_{\mathrm b}\) denotes the blue part \(2^jv\) for \(j\ge1\). Since parts in \(\mathcal E(n)\) are distinct and even parts may occur only in blue, each element of such a chain occurs at most once.

Choose the smallest odd integer \(v\) for which at least one part in the corresponding \(v\)-chain occurs in \(\lambda\). We define \(\Phi(\lambda)\) by changing only this chain and leaving all other chains unchanged. There are three cases.

\smallskip
\noindent\textbf{Case 1.} Exactly one of \(v_{\mathrm r}\) and \(v_{\mathrm b}\) occurs in \(\lambda\).  In this case, change its color:
\[
v_{\mathrm r}\longleftrightarrow v_{\mathrm b}.
\]
All other parts are left fixed.

\smallskip
\noindent\textbf{Case 2.} Both \(v_{\mathrm r}\) and \(v_{\mathrm b}\) occur in \(\lambda\).  Let \(a\ge1\) be the smallest integer such that \((2^a v)_{\mathrm b}\) does not occur in \(\lambda\). Thus
\[
v_{\mathrm r},\ v_{\mathrm b},\ (2v)_{\mathrm b},\ldots,(2^{a-1}v)_{\mathrm b}
\]
all occur in \(\lambda\), while \((2^a v)_{\mathrm b}\) does not. Replace this whole block by the single blue part \((2^a v)_{\mathrm b}\):
\[
v_{\mathrm r}+v_{\mathrm b}+(2v)_{\mathrm b}+\cdots+(2^{a-1}v)_{\mathrm b}
\longmapsto (2^a v)_{\mathrm b}.
\]

\smallskip
\noindent\textbf{Case 3.} Neither \(v_{\mathrm r}\) nor \(v_{\mathrm b}\) occurs in \(\lambda\).  By the choice of \(v\), some blue even part in the same chain occurs. Let \(c\ge1\) be the smallest integer such that \((2^c v)_{\mathrm b}\) occurs in \(\lambda\). Replace this part by the block
\[
(2^c v)_{\mathrm b}
\longmapsto
v_{\mathrm r}+v_{\mathrm b}+(2v)_{\mathrm b}+\cdots+(2^{c-1}v)_{\mathrm b}.
\]

The map is weight-preserving: Case 1 is immediate, and Cases 2 and 3 use
\[
v+v+2v+4v+\cdots+2^{r-1}v=2^r v
\qquad (r=a\text{ or }c).
\]
The image also remains in \(\mathcal E(n)\): all new even parts are blue, and the minimal choices in Cases 2 and 3 prevent repetitions.  Case 1 is self-inverse and Cases 2 and 3 are inverse to each other; hence \(\Phi^2=\mathrm{id}\).

In each case the number of red odd parts changes by one modulo \(2\). Therefore \(\Phi\) maps \(\mathcal E_0(n)\) onto \(\mathcal E_1(n)\) and maps \(\mathcal E_1(n)\) onto \(\mathcal E_0(n)\). Hence
\[
E_0(n)=E_1(n).
\]
Using Theorem~\ref{thm:E-total}, we also have
\[
E_0(n)+E_1(n)=E(n)=\overline p_o(n).
\]
Consequently,
\[
E_0(n)=E_1(n)=\frac{\overline p_o(n)}{2}.
\]
This completes the proof.
\end{proof}

\subsection*{Acknowledgements}
The authors were supported by the National Key R\&D Program of China (Grant No. 2024YFA1014500) and the National Natural Science Foundation of China (Grant No. 12201387).

\end{document}